\documentclass{biometrika}

\usepackage{times}
\usepackage{bm}
\usepackage{natbib, amsmath, url}

\usepackage[pagewise]{lineno}

\begin{document}
\nolinenumbers

\jname{Biometrika}
%% The year, volume, and number are determined on publication
\jyear{2013}
\jvol{100(2)}
\jnum{519-524 }
%% The \doi{...} and \accessdate commands are used by the production team
\doi{10.1093/biomet/ass084}
%%\accessdate{Advance Access publication}
%%\copyrightinfo{\Copyright\ 2011 Biometrika Trust\goodbreak {\em Printed in Great Britain}}

%% These dates are usually set by the production team
\iffalse
\received{ }
\revised{ }
\fi

%% The left and right page headers are defined here:
\markboth{T. Yan, \and J. Xu }{A central limit theorem in the
$\beta$-model}

%% Here are the title, author names and addresses
\title{A central limit theorem in the $\beta$-model for undirected random graphs
with  a diverging number of vertices}

\author{Ting Yan }
\affil{ Department of Statistics and Finance, \\
 University of Science and Technology of China, Hefei, Anhui 230026, P. R. China\\
\email{sunroom@mail.ustc.edu.cn}}

\author{\and Jinfeng Xu}
\affil{  Department of Statistics and Applied Probability, \\
National University of Singapore, 6 Science Drive 2, Singapore 117546, Singapore \\
\email{staxj@nus.edu.sg}}

\maketitle

\begin{abstract}
\citet{Chatterjee:Diaconis:Sly:2011} established the consistency of
the maximum likelihood estimator in the $\beta$-model for undirected
random graphs when the number of vertices goes to infinity. By
approximating the inverse of the Fisher information matrix, we
obtain its asymptotic normality under mild conditions. Simulation
studies and a data example illustrate the theoretical results.
\end{abstract}


\section{Introduction}
For an undirected random graph on $t$ vertices, the $\beta$-model
(\citeauthor{Chatterjee:Diaconis:Sly:2011}, 2011) assumes that there
exists an edge between vertices $i$ and $j$ with probability
\begin{equation*}\label{pij}
p_{i,j}=\frac{ e^{\beta_i+\beta_j} }{1+ e^{\beta_i+\beta_j} },~~  1\le i
\not= j\le t,
\end{equation*}
independently of all other edges, where $\beta_i$ is the influence
parameter of vertex $i$. First introduced by
\citet{Holland:Leinhardt:1981}  for directed networks, this model is
closely related to the Bradley--Terry model for rankings
(\citeauthor{Bradley:Terry:1952}, \citeyear{Bradley:Terry:1952}).
For undirected random graphs, it has been considered by
\citet{Newman:Strogatz:Watts:2001}, \citet{Jackson:2008}, and
\citet{Blitzstein:Diaconis:2011}. For many real world networks, the
number of vertices $t$ is  large and hence it is necessary to
consider asymptotics with $t\to\infty$. In the Bradley--Terry model
(\citeauthor{Bradley:Terry:1952}, \citeyear{Bradley:Terry:1952}) for
paired comparisons, \citet{Simons:Yao:1999} proved that the maximum
likelihood estimator is consistent and asymptotically normal when
the number of parameters goes to infinity. This contrasts with the
well-known Neyman--Scott problem under which the maximum likelihood
estimator fails even to attain consistency when the number of
parameters goes to infinity. More recently,
\citet{Chatterjee:Diaconis:Sly:2011} proved that the maximum
likelihood estimator of the $\beta$-model is consistent when $t$
goes to infinity.  In this note,  by approximating the inverse of
the Fisher information matrix, we further establish its asymptotic
normality under mild conditions.

\section{Main results}
Suppose that $\mathcal{G}$ is an undirected graph on $t$ vertices
generated from the $\beta$-model where $\beta=(\beta_1, \ldots,
\beta_t)^T\in R^t$ is unknown. Let $d_1, \ldots, d_t$ be the degrees
of the vertices of $\mathcal{G}$. The likelihood is
\begin{equation*}\label{likelihoodfunction}
\frac{ e^{\sum_i \beta_i d_i} }{ \prod_{i<j}
(1+e^{\beta_i+\beta_j})}.
\end{equation*}
The maximum likelihood estimator $\hat{\beta}$ of $\beta$ can be
obtained by solving the equations
\begin{equation}\label{moment}
d_i=\sum_{j\neq i} \frac{e^{\hat{\beta}_i+\hat{\beta}_j}}{ 1+
e^{\hat{\beta}_i+\hat{\beta}_j} },~~ i=1,\ldots,t.
\end{equation}
In a preprint available at \url{http://arxiv.org/abs/1105.6145},
Rinaldo, Petrovic, and Fienberg obtained necessary and sufficient
conditions for the existence and uniqueness of $\hat{\beta}$.
\citet{Chatterjee:Diaconis:Sly:2011} established the following
theorem:
\begin{theorem}\label{dense-con}
Define $L_t=\max_{1\le i\le t}|\beta_i|$.\\
(a) If $L_t=o(\log t)$, then with probability tending to one as
$t\to\infty$,
 there exists a unique solution $\hat{\beta}$ of the maximum
likelihood equations {\rm (1)}.\\
(b) If $L_t=o\{\log (\log t)\}$, then
\begin{equation*}
\max_{1\le i\le t} |\hat{\beta}_i-\beta_i|\le O_p\{ (\log
t)^{1/2}t^{-1/2}e^{c_1e^{c_2L_t}+c_3L_t}  \}=o_p(1),
\end{equation*}
where $c_1, c_2$ and $c_3$ are positive constants. Hence
$\hat{\beta}$ is uniformly consistent.
\end{theorem}

Denote the covariance matrix of $d=(d_1, \ldots, d_t)$ by
$V_t=(v_{i,j})_{t\times t}$, where
\begin{equation*}\label{V}
v_{i,j}=\frac{ e^{\beta_i+\beta_j} }{(1+e^{\beta_i+\beta_j})^2},~~
v_{i,i}=\sum_{j\neq i} v_{i,j} ~~(i,j=1,\ldots, t;i\neq j).
\end{equation*}
This is also the Fisher information matrix for $\beta$. To establish
the asymptotic normality of $\hat{\beta}$, we need an accurate
approximation to $V_t^{-1}$. Let $S_t=(s_{i,j})_{t\times t}$, where
$s_{i,j}=\delta_{i,j}/v_{i,i}-1/v_{\cdot\cdot}$, $\delta_{i,j}$ is
the Kronecker delta function and $v_{\cdot\cdot}=\sum_{i,j=1;i\neq
j}^tv_{i,j}$. In Proposition 1 which is given in Appendix 1, we
obtain an upper bound on the error of using $S_t$ to approximate
$V^{-1}_t$.  In the following, we present a central limit theorem
for the maximum likelihood estimator in the $\beta$ model. The proof
is given in Appendix 2.

\begin{theorem}\label{central-dense}
If $L_t=o\{\log (\log t)\}$, then for any fixed $r\ge 1$, as
$t\to\infty$, the vector consisting of the first $r$ elements of
$G_t^{1/2}(\hat{\beta}-\beta)$ is asymptotically standard
multivariate normal, where $G_t=\mbox{diag}(v_{1,1}, \ldots,
v_{t,t})$ and $G_t^{1/2}=\mbox{diag}(v_{1,1}^{1/2}, \ldots,
v_{t,t}^{1/2})$.
\end{theorem}

\begin{remark}
By Theorem 2, for any fixed $i$, as $t\rightarrow \infty$, the
convergence rate of $\hat{\beta}_i$ is $1/v_{i,i}^{1/2}$. Since
$(t-1)e^{-2L_t}/4\le v_{i,i}\le (t-1)/4$, the rate of convergence is
between $O(t^{-1/2}e^{L_t})$ and  $O(t^{-1/2})$.
\end{remark}

\section{Numerical examples}
We conduct simulation studies to illustrate our theoretical results.
By Theorem 2, we construct approximate $95\%$ confidence intervals
for $\beta_i$ and $\beta_i-\beta_j$. We report the coverage
probabilities for certain $\beta_i-\beta_j$ and the average coverage
probabilities for $\beta_i ~(i=1, \ldots, t)$ as well as the
probabilities that the maximum likelihood estimator does not exist.
Let $\beta_i=iL_t/t$ and choose $L_t=0, \log (\log t), (\log
t)^{1/2}$ or $\log t$. Using $10,000$ simulations for each scenario,
the results are summarized in Table \ref{simulation1}. We see that
when $L_t=0$ or $\log(\log t)$, the coverage probabilities are very
close to the nominal level, indicating the adequacy of the
confidence intervals. When $L_t=(\log t)^{1/2}$ or $\log t$, the
maximum likelihood estimator does not exist with nonzero probability
and the coverage probabilities deviate much from the nominal level.
Using the normal Q-Q plots, when $L_t=0$ or $\log(\log t)$, the
normality of the estimator is quite evident. However, when
$L_t=(\log t)^{1/2}$, there is a notable deviation from normality.
That demonstrates that the condition on $L_t$ in Theorem 2 is
critical in ensuring the existence of the maximum likelihood
estimator and its asymptotic normality.

\begin{table}[h]\centering
\caption{Estimated coverage probabilities  and probabilities that
the maximum likelihood estimator does not exist (in parentheses),
both multiplied by $100$} \label{simulation1} \scriptsize
\begin{tabular}{ccccccc}
&&&&&&\\
t       &  $(i,j)$ & $L_t=0$ & $L_t=\log (\log t)$ & $L_t=(\log t)^{1/2}$ & $L_t=\log t$ \\
&&&&&&\\
50           &(1,50)      &$ 94.6~ ( 0 )$&$ 95.8~ ( 0.1 )$&$ 89.4 ~( 8)$&$ 0 ~( 100 )$\\
             &(25,26)     &$ 95.0~ ( 0 )$&$ 95.5~ ( 0.1 )$&$ 88.4~ ( 8 )$&$ 0~ ( 100 )$\\
             &(49,50)     &$ 95.2~ ( 0 )$&$ 95.4~ ( 0.1 )$&$ 91.6~ ( 8 )$&$ 0~ ( 100 )$\\
             &ACP       &$ 95.1~ ( 0 )$&$ 95.4~ ( 0.1 )$&$ 88.4~ ( 8 )$&$ 0~ ( 100 )$ \\
&&&&&&\\
100         &(1,100)     &$ 94.3~ ( 0 )$&$ 95.1~ ( 0 )$&$ 97.0~ ( 0.5 )$&$ 0~ ( 100 )$\\
            &(50,51)     &$ 94.6~ ( 0 )$&$ 95.4~ ( 0 )$&$ 95.1~ ( 0.5 )$&$ 0~ ( 100 )$ \\
            &(99,100)    &$ 94.8~ ( 0 )$&$ 95.7~ ( 0 )$&$ 97.7~ ( 0.5 )$&$ 0~ ( 100 )$\\
            &ACP         &$ 95.0~ ( 0 )$&$ 95.2~ ( 0 )$&$ 95.2~ ( 0.5 )$&$ 0~ ( 100 )$\\
&&&&&&\\
200         &(1,200)     &$ 94.9~ ( 0 )$&$ 95.1~ ( 0 )$&$ 96.1~ ( 0 )$&$ 0~ ( 100 )$\\
            &(100,101)   &$ 95.3~ ( 0 )$&$ 95.0~ ( 0 )$&$ 95.1~ ( 0 )$&$ 0~ ( 100 )$ \\
            &(199,200)   &$ 95.1~ ( 0 )$&$ 95.2~ ( 0 )$&$ 96.5~ ( 0 )$&$ 0~ ( 100 )$ \\
            &ACP         &$ 95.1~ ( 0 )$&$ 95.1~ ( 0 )$&$ 95.3~ ( 0 )$&$ 0~ ( 100 )$\\
&&&&&&\\
\end{tabular}
\begin{tabnote}
\mbox{$(i,j)$, coverage probability for $\beta_i-\beta_j$; ACP,
average coverage probability for $\beta_1,\ldots, \beta_t$.}
\end{tabnote}
\end{table}

We analyze the food web dataset in \citet{Blitzstein:Diaconis:2011},
which contains $33$ organisms in Chesapeake Bay, each represented by
a vertex in the graph. As in \citet{Blitzstein:Diaconis:2011}, we
study the simple graph after omitting the self-loop at vertex 19.
The influence parameters and their standard errors are reported in
Table 2. The largest four degrees are 8, 8, 10, 9 for vertices 2, 7,
8, 22, which also have the largest four influence parameters
$-0.083, -0.083, 0.275, 0.102$ from Table \ref{merits}. On the other
hand, the four vertices with the smallest influence parameter
$-2.602$ all have degree $1$.

\begin{table}[h]\centering
\caption{The food web dataset: the estimated influence parameters
$\hat{\beta}$ and their standard errors (in parentheses)}
\label{merits}  \scriptsize
\begin{tabular}{cccccccc }
&&&&&&&\\
Vertex &  $\hat{\beta}$       & Vertex   & $\hat{\beta}$  &  Vertex    &  $\hat{\beta}$    &  Vertex     &  $\hat{\beta}$ \\
&&&&&&&\\
~~1 & $-0.29~ (     2.23 )$   &~~2   & $-0.08~ (     2.33 )$   & ~~3   &  $-0.75~ (     1.98 )$  & ~~4 & $-2.60~ (     0.98 )$     \\
~~5 & $-2.60~ (     0.98 )$   &~~6   & $-1.85~ (     1.35 )$   & ~~7   &  $-0.08~ (     2.33 )$  & ~~8 & ~~~$0.28~  (     2.49 )$    \\
~~9 & $-1.04~ (     1.82 )$   &10  & $-1.85~ (     1.35 )$   & 11  &  $-1.04~  (     1.82)$  &12 & $-0.75~ (     1.98 )$  \\
13& $-1.39~ (     1.61 )$   &14  & $-0.51~ (     2.12 )$   & 15  &  $-0.29~ (     2.23)$  &16 & $-1.39~ (     1.61 )$ \\
17& $-1.85~ (     1.35 )$   &18  & $-0.29~ (     2.23 )$   & 19  &  $-0.51~  (     2.12)$  &20 & $-2.60~ (     0.98 )$ \\
21& $-1.85~ (     1.35 )$   &22  & ~~~$0.10~  (     2.42 )$   & 23  &  $-0.51~  (     2.12)$  &24 & $-2.60~ (     0.98 )$  \\
25& $-1.39~ (     1.61 )$   &26  & $-1.04~ (     1.82 )$   & 27  &  $-0.51~  (     2.12)$  &28 & $-1.39~ (     1.61 )$\\
29& $-1.39~ (     1.61 )$   &30  & $-1.39~ (     1.61 )$   & 31  &  $-1.85~ (     1.35)$  &32 & $-1.04~ (     1.82 )$\\
33& $-1.04~ (     1.82 )$   &    &                        &     &                        &   &                  \\
&&&&&&&\\
\end{tabular}
\begin{tabnote}
\begin{center}
$d=(7, 8, 5, 1, 1, 2, 8, 10, 4, 2, 4, 5, 3, 6, 7, 3, 2, 7, 6, 1, 2,
9, 6, 1, 3, 4, 6, 3, 3, 3, 2, 4, 4)$.
\end{center}
\end{tabnote}
\end{table}

\centerline{\normalsize \textsc{Acknowledgement} } We are grateful
to the editor, the associate editor and a referee for helpful
comments. This research was supported by a grant from National
University of Singapore.
\appendix
\section*{Appendix 1}
\begin{proposition}\label{pro0}
As $t\rightarrow \infty$,
\begin{equation}\label{O-upperbound}
||V^{-1}_t-S_t||\le O\left\{ \frac{e^{6L_t}}{(t-1)^2} \right\},
\end{equation}
where $||A||=\max_{i,j} |a_{i,j}|$ for a matrix $A=(a_{i,j})$.
\end{proposition}

\begin{proof}[of Proposition~\ref{pro0}]
Define $m=\min_{1\le i<j\le t}v_{i,j}$ and $M=\max_{1\le i<j\le
t}v_{i,j}$. It is easy to see that
\begin{equation}\label{vij-ineq}
\frac{e^{2L_t}}{(1+e^{2L_t})^2}\le
v_{i,j}=\frac{e^{\beta_i+\beta_j}}{(1+e^{\beta_i+\beta_j})^2} \le
\frac{1}{4} ~~(i\neq j).
\end{equation}
By \eqref{vij-ineq}, we have $m\ge e^{2L_t}/(1+e^{2L_t})^2$ and
$M\le 1/4$. Denote the $t\times t$ identity matrix by $I_t$. Write
$F_t=(f_{i,j})=V_t^{-1}-S_t$, $R_t=(r_{i,j})=I_t-V_tS_t$ and
$W_t=(w_{i,j})=S_tR_t$. We have the recursion
\[F_t=(V_t^{-1}-S_t)(I_t-V_tS_t)+S_t(I_t-V_tS_t)=F_tR_t+W_t,\]
and it follows that, for any $i$,
\[
f_{i,j}=\sum_{k=1}^t
f_{i,k}\{(\delta_{k,j}-1)\frac{v_{k,j}}{v_{j,j}}+\frac{2v_{k,k}}{v_{\cdot\cdot}}\}
+w_{i,j} ~~~~~(j=1,\ldots, t).
\]
Fixing $i$, let $f_{i,\alpha}=\max_{1\le k\le t}f_{i,k}$ and
$f_{i,\beta}=\min_{1\le k\le t} f_{i,k}$.   Since $2\sum_{k=1}^t
f_{i,k}v_{k,k} = 1$, we have
 $f_{i,\beta}\le 1/(2v_{\cdot\cdot})$ and $f_{i,\alpha}\ge 0$.
By direct calculation, it can be shown that for all $i,j,k$,
\begin{equation}\label{wijin}
\max(|w_{i,j}|, |w_{i,j}-w_{i,k}|)\le \frac{M}{m^2(t-1)^2},
\end{equation}
and
\begin{equation}\label{falphabeta}
f_{i,\alpha}-f_{i,\beta}=\sum_{k=1}^t
(f_{i,k}-f_{i,\beta})
\{(1-\delta_{k,\beta})\frac{v_{k,\beta}}{v_{\beta,\beta}}-(1-\delta_{k,\alpha})\frac{v_{k,\alpha}}{v_{\alpha,\alpha}}\}
+w_{i,\alpha}-w_{i,\beta}.
\end{equation}
 Define $a=M/\{m^2(t-1)^2\}$,
$\Omega=\left\{k:(1-\delta_{k,\beta})v_{k,\beta}/v_{\beta,\beta} \ge
(1-\delta_{k,\alpha})v_{k,\alpha}/v_{\alpha,\alpha}\right\}$ and
$|\Omega|=\lambda$. It follows that
\begin{eqnarray}\nonumber
\sum_{k\in \Omega}(f_{i,k}-f_{i,\beta})\left\{
(1-\delta_{k,\beta})\frac{v_{k,\beta}}{v_{\beta,\beta}}-(1-\delta_{k,\alpha})\frac{v_{k,\alpha}}{v_{\alpha,\alpha}}\right\}
&\le & (f_{i,\alpha}-f_{i,\beta}) \left\{\frac{\sum_{k\in
\Omega}v_{k,\beta}}{v_{\beta,\beta}}-\frac{\sum_{k\in
\Omega}(1-\delta_{k,\alpha})v_{k,\alpha}}{v_{\alpha,\alpha}}\right\} \\
\label{yyy} & \le & (f_{i,\alpha}-f_{i,\beta})f(\lambda),
\end{eqnarray}
where $f(\lambda)=\lambda M /\{ \lambda M+ (t-1-\lambda)m \}-
 (\lambda-1)m/\{ (\lambda-1)m+ (t-\lambda)M\}$. Note that
$f(\lambda)$ takes its maximum at $\lambda=t/2$ when $\lambda\in [1,
t-1]$ and  $f(t/2)=\{ tM-(t-2)m\}/\{ tM+(t-2)m\}$. By \eqref{wijin},
\eqref{falphabeta}, and \eqref{yyy},
\begin{equation*}
f_{i,\alpha}-f_{i,\beta}\le \frac{ tM-(t-2)m}{ tM+(t-2)m}\times
(f_{i,\alpha}-f_{i,\beta})+a.
\end{equation*}
Hence
\begin{equation*}
f_{i,\alpha}-f_{i,\beta} \le  \frac{
M\{tM+(t-2)m\}}{2(t-2)m^3(t-1)^2}.
\end{equation*}
Since $f_{i,\alpha}=\max_{k}f_{i,k}$ and $f_{i,\beta}=\min_{k}
f_{i,k}$, we have $\max_{1\le k\le t} |f_{i,k}| \le f_{i,\alpha} -
f_{i,\beta} + f_{i,\beta}I(f_{i,\beta}>0),$ where $I(\cdot)$ is the
indicator function. Hence,
\begin{eqnarray*}
 \max_{1
 \le k\le t} |f_{i,k}|  \le  \frac{ M(tM+(t-2)m)}{2(t-2)m^3(t-1)^2}+\frac{1}{2v_{\cdot\cdot}}\le \frac{ M(tM+(t-2)m)}{2(t-2)m^3(t-1)^2}+\frac{1}{2m(t-1)^2}
 \le  O\left\{\frac{e^{6L_t}}{(t-1)^2} \right\}.
\end{eqnarray*}
\end{proof}

\section*{Appendix 2}

 Let $d_{i,j}=1$
if there exists an edge between vertices $i$ and $j$ and $0$
otherwise. Note that $d_i=\sum_{j\neq i}d_{i,j}$ and $\sum_i d_i/2
=\sum_{1\le i<j \le t} d_{i,j}$ are sums of $t-1$ and $t(t-1)/2$
independent Bernoulli random variables, respectively. By the central
limit theorem for the bounded case in \citeauthor{Loeve:1977}
(\citeyear{Loeve:1977}, p. 289), we know that $v_{i,i}^{-1/2} \{d_i
- E(d_i)\}$ and $(2v_{\cdot\cdot})^{-1/2}[\sum_i \{d_i - E(d_i)\}]$
are asymptotically standard normal if $v_{i,i}$ diverges. By
\eqref{vij-ineq}, we have
\[\frac{(t-1)e^{2L_t}}{(1+e^{2L_t})^2}\le v_{i,i} \le \frac{t-1}{4},
~~~ i=1, \ldots, t;~~ v_{\cdot\cdot}\ge
\frac{t(t-1)e^{2L_t}}{(1+e^{2L_t})^2}.
\] If $e^{L_t}=o( t^{1/2} )$, then
\[
v_{\cdot\cdot}^{-1} \max_{i=1,\ldots, t} v_{i,i} \le
(1+e^{2L_t})^2/(4te^{2L_t})=o(1),
\]
and $v_{i,i}^{1/2}[S_t\{d-E(d)\}]_i =v_{i,i}^{-1/2} \{d_i -
E(d_i)\}+o_p(1)$. Thus, we have the following proposition.

\begin{proposition}\label{pro1}
If $e^{L_t}=o( t^{1/2} )$, then for any fixed $r\ge 1$, as
$t\to\infty$, the vector consisting of the first $r$ elements of
$S_t\{d-E(d)\}$ is asymptotically multivariate normal with mean zero
and covariance matrix $(G_t^{-1})_{r\times r}$, where
$G_t^{-1}=\mbox{diag}(v_{1,1}^{-1}, \ldots, v_{t,t}^{-1})$.
\end{proposition}

\begin{lemma}\label{central1-lemma1}
Let $F_t=V_t^{-1}-S_t$ and $U_t=\mbox{cov}[F_t
\{d-E(d)\}]$. Then
\begin{equation}
||U_t||\le ||V_t^{-1}-S_t||+\frac{(1+e^{2L_t})^4}{4e^{4L_t}(t-1)^2}.
\end{equation}
\end{lemma}
\begin{proof}
Note that
\begin{equation*}
U_t=F_tV_tF^T_t=(V_t^{-1}-S_t)-S_t (I_t-V_tS_t),
\end{equation*}
and
\begin{equation*}
\{S_t(I_t-V_tS_t)\}_{i,j}=\frac{(\delta_{i,j}-1)v_{i,j}}{v_{i,i}v_{j,j}}+\frac{1}{v_{\cdot\cdot}}.
\end{equation*}
By  \eqref{wijin},
\begin{equation*}
|\{S_t(I_t-V_tS_t)\}_{i,j}|\le
\max\{\frac{(1+e^{2L_t})^4}{4e^{4L_t}(t-1)^2},
\frac{(1+e^{2L_t})^2}{t(t-1)e^{2L_t}} \}\le
\frac{(1+e^{2L_t})^4}{4e^{4L_t}(t-1)^2},
\end{equation*}
Thus,
\begin{eqnarray*}
||U_t||  \le  ||V_t^{-1}-S_t||+||S_t(I_t-V_tS_t)||
\le
||V_t^{-1}-S_t||+\frac{(1+e^{2L_t})^4}{4e^{4L_t}(t-1)^2}.
\end{eqnarray*}
\end{proof}

\begin{lemma}\label{central-lemma2}
Assume that Theorem 1 (b) holds. If $L_t=o\{\log (\log t)\}$, then
for any $i$,
\begin{equation}
\hat{\beta}_i-\beta_i=[V_{t}^{-1}\{d-E(d)\}]_i+o_p(t^{-1/2}).
\end{equation}
\end{lemma}
\begin{proof}
By Theorem 1 (b), we know that
\[
\lambda_t=\max_{1\le i\le t} |\hat{\beta}_i-\beta_i|=O_p\{(\log
t)^{1/2}t^{-1/2}e^{c_1e^{c_2L_t}+c_3L_t} \}.
\]
Let $\hat{\gamma}_{i,j}=\hat{\beta}_i + \hat{\beta}_j -\beta_i
-\beta_j$. By Taylor expansion, for any $i\not=j$,
\[
\frac{ e^{\hat{\beta}_i +\hat{\beta}_j } }{ 1+
e^{\hat{\beta}_i+\hat{\beta}_j} }
-\frac{e^{\beta_i+\beta_j}}{1+e^{\beta_i+\beta_j}} =
 \frac{e^{\beta_i+\beta_j}}{(1+e^{\beta_i+\beta_j})^2}\hat{\gamma}_{ij}
 +h_{i,j},
\]
where
$$h_{i,j}=\frac{e^{\beta_i+\beta_j+\theta_{i,j}\hat{\gamma}_{i,j}}(1-e^{\beta_i+\beta_j+\theta_{i,j}\hat{\gamma}_{i,j}})  }{2(1+e^{\beta_i+\beta_j+\theta_{i,j}\hat{\gamma}_{ij}})^3}\hat{\gamma}_{i,j}^2,$$
and $0\le\theta_{i,j}\le 1$. Rewrite \eqref{moment} as
\begin{equation*}
d-E(d)=V_t(\hat{\beta}-\beta)+h,
\end{equation*}
where ${h}=(h_1, \ldots, h_t)^T$ and $h_i=\sum_{j\neq i}h_{i,j}$.
Equivalently,
\begin{equation}\label{represent}
\hat{\beta}-\beta=V_t^{-1}\{d-E(d)\}+V_t^{-1}h.
\end{equation}
Since $|e^{x}(1-e^{x})/(1+e^x)^3| \le 1$, we have
\[
|h_{i,j}|\le |\hat{\gamma}_{i,j}^2|/2\le 2\lambda_t^2,~~~~
|h_i|\le \sum_{j\neq i}|h_{i,j}|\le 2(t-1)\lambda_t^2.\]
 Note that
$(S_t {h})_i  =  h_i/v_{i,i} -v_{\cdot\cdot}^{-1}\sum_{j=1}^t
h_{j},\mbox{~~and~~} (V_t^{-1} {h})_i=(S_t {h})_i+(F_t {h})_i$. By
direct calculation, we have
\begin{equation*}
|(S_t {h})_i|  \le \frac{8\lambda_t^2(1+e^{2L_t})^2}{e^{2L_t}} = O\{
(\log t)t^{-1} e^{2c_1e^{c_2L_t}+(2c_3+2)L_t}  \},
\end{equation*}
and, by Proposition 1,
\begin{equation*}
|(F_t h)_i| \le  ||F_t||\times (t\max_i|h_i|) \le  O(e^{6L_t} \times
\lambda_t^2 )\le O\{(\log t)t^{-1} e^{2c_1e^{c_2L_t}+(2c_3+6)L_t}
\}.
\end{equation*}
If $L_t=o\{\log (\log t) \}$, then $|(V_t^{-1}h)_i|\le
|(S_th)_i|+|(F_th)_i|=o(t^{-1/2})$. This completes the proof.
\end{proof}

\begin{proof}[of Theorem~\ref{central-dense}]
By \eqref{represent},
\begin{equation*}
(\hat{\beta}-\beta)_i= [S_t\{d-E(d)\}]_i+ [F_t\{d-E(d)\}]_i + (V_t^{-1}h)_i.
\end{equation*}
By Lemmas 1 and 2, if $L_t=o\{\log (\log t)\}$, then
\[(\hat{\beta}-\beta)_i=
[S_t\{d-E(d)\}]_i+o(t^{-1/2}).\] Theorem \ref{central-dense} follows
directly from Proposition \ref{pro1}.
\end{proof}


\begin{thebibliography}{}

\bibitem[\protect\astroncite{Bradley \&\ Terry}{1952}]{Bradley:Terry:1952}
{\sc Bradley, R. A. \&\ Terry, M. E.} (1952).
\newblock Rank analysis of incomplete block designs I. The method of paired
comparisons.
\newblock {\em Biometrika} {\bf 39}, 324--345.


\bibitem[\protect\astroncite{Blitzstein \&\ Diaconis}{2011}]{Blitzstein:Diaconis:2011}
{\sc Blitzstein, J. \&\ Diaconis, P.} (2011).
\newblock
A sequential importance sampling algorithm for generating random graphs with prescribed degrees.
\newblock{\em Internet Mathematics}  {\bf 6}, 489--522.


\bibitem[\protect\astroncite{Chatterjee et~al.}{2011}]{Chatterjee:Diaconis:Sly:2011}
{\sc Chatterjee, S., Diaconis, P. \&\ Sly, A.} (2011).
\newblock Random graphs with a given degree sequence.
\newblock{\em Ann. Appl. Probab.} {\bf 21}, 1400--1435.

\bibitem[\protect\astroncite{Holland \&\ Leinhardt}{1981}]{Holland:Leinhardt:1981}
{\sc
Holland, P. ~W. \&\ Leinhardt, S.} (1981).
\newblock
An exponential family of probability distributions for random graphs
(with discussion),
\newblock{\em J. Amer. Stat. Assoc.} {\bf 76}, 33--50.


\bibitem[\protect\astroncite{Jackson}{2008}]{Jackson:2008}
{\sc Jackson, M.~ O.} (2008).
\newblock{\em Social and Economic Networks}.
\newblock
Princeton, New Jersey: Princeton University Press.

\bibitem[\protect\astroncite{Lo\`{e}ve}{1977}]{Loeve:1977}
{\sc Lo\`{e}ve, M.} (1977).
\newblock{\em Probability Theory}. 4th Ed.
\newblock
New York: Springer-Verlag.



\bibitem[\protect\astroncite{Newman et~al.}{2001}]{Newman:Strogatz:Watts:2001}
{\sc Newman, M. E. J., Strogatz, S. H. \&\ Watts D. J.} (2001).
\newblock Random graphs with arbitrary degree distributions
and their applications,
\newblock{\em Phys. Rev. E.} {\bf 64}, 026118.


\bibitem[\protect\astroncite{Simons \&\ Yao}{1999}]{Simons:Yao:1999}
{\sc Simons, G. \&\ Yao, Y.} (1999).
\newblock
Asymptotics when the number of parameters tends to infinity in the
Bradley--Terry model for paired comparisons.
\newblock{\em Ann. Stat.}  {\bf 27}, 1041--1060.
\end{thebibliography}
\end{document}